\documentclass[11pt]{amsart}

\usepackage[T1]{fontenc}
\usepackage[utf8]{inputenc}
\usepackage{lmodern}
\usepackage{amsmath,amssymb,amsthm,mathtools}
\usepackage[a4paper,margin=1in]{geometry}
\usepackage{microtype}
\usepackage{xcolor}
\usepackage[colorlinks=true,linkcolor=blue!50!black,citecolor=blue!50!black,urlcolor=blue!50!black]{hyperref}
\usepackage{verbatim}

\numberwithin{equation}{section}

\parskip=\smallskipamount

\theoremstyle{plain}
\newtheorem{theorem}{Theorem}[section]

\newtheorem{corollary}[theorem]{Corollary}
\newtheorem{problem}[theorem]{Problem}
\newtheorem{lemma}[theorem]{Lemma}

\theoremstyle{remark}
\newtheorem{remark}[theorem]{Remark}

\begin{document} 

%
%
\title[On an extremal problem for harmonic maps conformal at a point]{On an extremal problem for harmonic maps \\ conformal at a point}

\author{Franc Forstneri{\v c} and David Kalaj}

\address{Franc Forstneri{\v c}, Faculty of Mathematics and Physics, University of Ljubljana, Jadranska 19, 1000 Ljubljana, Slovenia}

\address{Franc Forstneri{\v c}, Institute of Mathematics, Physics, and Mechanics, Jadranska 19, 1000 Ljubljana, Slovenia}

\email{franc.forstneric@fmf.uni-lj.si}

\address{David Kalaj, Faculty of Natural Sciences and Mathematics, University of Montenegro, Cetinjski put bb, 81000 Podgorica, Montenegro} 

\email{davidk@ucg.ac.me}





\date{9 May 2026. This version: 14 Seotember 2026}


\begin{abstract} Let \(\mathbb D\) denote the unit disc in \(\mathbb C\). For a domain \(D\subset\mathbb C\) and a point \(p\in D\), let \(M_D(p)\) denote the supremum of \(\|df_0\|\) over all harmonic maps \(f:\mathbb D\to D\) with \(f(0)=p\) whose differential \(df_0\) at \(0\in\mathbb D\) is conformal. If \(f:\mathbb D\to D\) is a conformal diffeomorphism onto \(D\) with \(f(0)=p\), then \(\|df_0\|\le M_D(p)\). In a recent paper, the authors proved that equality holds when \(D=\mathbb D\), and they asked whether equality can hold only when \(D\) is a round disc. We give a negative answer by proving that, among bounded convex pointed domains \(p\in D\subset\mathbb C\), and up to translations, rotations, and reflections, equality holds if and only if, after moving \(p\) to the origin, \(D=F(\mathbb D)\), where \(F:\mathbb D\to\mathbb C\) is a holomorphic map with \(F(0)=0\) and 
\[ 
	F'(z)=\frac{c}{1+az+\lambda z^2}, 
\] 
where \(c>0\), \(|\lambda|<1\), and \(|a-\bar a\lambda|<1-|\lambda|^2\). 
This family contains strongly convex examples that are not round discs.
\bigskip 

\noindent 
R\'ESUM\'E. 
Soit \(\mathbb D\) le disque unit\'e de \(\mathbb C\). Pour un domaine \(D\subset\mathbb C\) et un point \(p\in D\), notons \(M_D(p)\) le supr\'emum de \(\|df_0\|\) sur l'ensemble des applications harmoniques \(f:\mathbb D\to D\) telles que \(f(0)=p\) et dont la diff\'erentielle \(df_0\) en \(0\in\mathbb D\) est conforme. Si \(f:\mathbb D\to D\) est un diff\'eomorphisme conforme de \(\mathbb D\) sur \(D\) tel que \(f(0)=p\), alors \(\|df_0\|\le M_D(p)\). Dans un article r\'ecent, les auteurs ont d\'emontr\'e que l'\'egalit\'e est atteinte lorsque \(D=\mathbb D\), et ils se sont demand\'e si elle ne pouvait l'\^etre que lorsque \(D\) est un disque rond. Nous apportons une r\'eponse n\'egative en d\'emontrant que, parmi les domaines convexes point\'es et born\'es \(p\in D\subset\mathbb C\), et \`a translations, rotations et r\'eflexions pr\`es, l'\'egalit\'e est atteinte si et seulement si, apr\`es avoir envoy\'e \(p\) \`a l'origine par une translation, \(D=F(\mathbb D)\), o\`u \(F:\mathbb D\to\mathbb C\) est une application holomorphe telle que \(F(0)=0\) et \[ F'(z)=\frac{c}{1+az+\lambda z^2}, \] o\`u \(c>0\), \(|\lambda|<1\) et \(|a-\bar a\lambda|<1-|\lambda|^2\). Cette famille contient des exemples fortement convexes qui ne sont pas des disques ronds.
\end{abstract}

\maketitle

%
%
%
%
\section{Introduction}

The classical Schwarz lemma and its Schwarz--Pick refinement control the
derivative of a holomorphic self-map of the unit disc
$\mathbb D=\{z\in \mathbb C: |z|<1\}$ and characterise the
extremal maps as automorphisms of the disc.
The analogous extremal problems for harmonic maps are
more subtle, especially when one imposes conformality only at a single point.
We refer to Duren~\cite{Duren}, Forstneri\v{c} and Kalaj~\cite{ForstnericKalaj2024},
and Kalaj~\cite{Kalaj2003,Kalaj2019} for related results and background.

The motivation for the present paper is the problem posed by the authors in \cite[Problem~4.3]{ForstnericKalaj2024}. Let \(D\) be a connected and
simply connected domain in the complex plane $\mathbb C$. Given
a point \(p\in D\), set
\begin{equation}\label{eq:MD}
	M_D(p) =
	\sup\bigl\{\|df_0\|:\ f:\mathbb D\to D \text{ harmonic},\ f(0)=p,\ df_0
	\text{ conformal}\bigr\}.
\end{equation}
If \(f:\mathbb D\to D\) is a holomorphic map with \(f(0)=p\), then clearly
\begin{equation}\label{eq:equality}
	|f'(0)| \le M_D(p).
\end{equation}
Assuming that $D$ is a bounded convex domain, we prove
in Theorem \ref{th:existence} that there exists an extremal harmonic map
\(f:\mathbb D\to D\) with $f(0)=p$ such that $\|df_0\|=M_D(p)$.

In the special case when $D$ is a round disc, the extremal maps for which
equality holds in \eqref{eq:equality} are conformal diffeomorphisms
of $\mathbb D$ onto $D$; see \cite[Theorem 1.1]{ForstnericKalaj2024}. 
In particular, conformality
of a harmonic map \(f:\mathbb D\to \mathbb D\) at the point $0$,
together with extremality of $\|df_0\|$, implies conformality of
$f$ at all points. This is a nontrivial generalisation of the classical
Schwarz--Pick lemma, whose proof in \cite{ForstnericKalaj2024} is
inspired by the seminal work of Lempert \cite{Lempert1981} on geodesics
of the Kobayashi metric on convex domains in $\mathbb C^n$.
The analogous result holds for harmonic maps
$\mathbb D\to \mathbb B^n$
to the ball in the real Euclidean space $\mathbb R^n$
for any $n\ge 3$, and the extremal maps are conformal harmonic
diffeomorphisms from $\mathbb D$
onto proper affine linear discs in $\mathbb B^n$;
see \cite[Theorem 2.1]{ForstnericKalaj2024}. These results led
to the introduction of a Kobayashi-type pseudometric on domains
in $\mathbb R^n$, $n\ge 3$, called the \emph{minimal metric}
since it pertains to conformal minimal surfaces;
see \cite[Sect.\ 6]{ForstnericKalaj2024} and \cite{DrinovecForstneric2023PAMQ}. 
By \cite[Theorem 6.2]{ForstnericKalaj2024}, 
the minimal metric on the ball $\mathbb B^n$ is the
Cayley--Klein metric, one of the classical models of hyperbolic geometry.

By \cite[Example 4.2]{ForstnericKalaj2024},
conformal diffeomorphisms $\mathbb D\to D$ onto more general
domains $D\subset \mathbb C$ fail to reach the extremal
value in \eqref{eq:equality} even if $D$ is convex.
The question raised in~\cite[Problem~4.3]{ForstnericKalaj2024} is 
whether a conformal diffeomorphism can attain the extremal value in 
\eqref{eq:equality} only when $D$ is a round disc.
We identify a family of bounded strongly convex pointed domains
$p\in D\subset \mathbb C$ for which a conformal diffeomorphism
$\mathbb D\to D$ satisfies equality in \eqref{eq:equality}.
Such domains will be called exceptional.

%
%
\begin{theorem}\label{th:quadratic-exceptional-family}
Let \(a,\lambda\in\mathbb C\) be such that
the quadratic polynomial
$
	q(z):=1+az+\lambda z^2
$
has no zeros on \(\overline{\mathbb D}\). Equivalently, by the Schur--Cohn
criterion (see Lemma \ref{lem:schur-cohn-quadratic}),
\begin{equation}\label{eq:SC}
	|\lambda|<1
	\quad \text{and}\quad
	|a-\bar a\lambda|<1-|\lambda|^2.
\end{equation}
Given $c>0$, let \(F:\mathbb D\to\mathbb C\) be the holomorphic
map determined by the conditions
\begin{equation}\label{eq:F}
	F(0)=0	\quad \text{and}\quad
	F'(z)=\frac{c}{1+az+\lambda z^2}.
\end{equation}
Then $F$ is a biholomorphic map onto the bounded strongly convex domain
\(D=F(\mathbb D)\), and $M_D(0)=|F'(0)|=c$. The harmonic extremal maps
attaining \(M_D(0)\) are
\[
	z\longmapsto F(e^{i\theta}z)
	\quad\text{and}\quad
	z\longmapsto F(e^{i\theta}\bar z),
	\qquad \theta\in\mathbb R.
\]
\end{theorem}

We obtain a bigger family of exceptional domains by applying
translations, rotations, and reflections to the maps $F$
in the above theorem. However, postcomposition with a nonrigid map does not
preserve the class of harmonic maps
used in the definition of $M_D(p)$. In particular, we are not allowed
to use holomorphic automorphisms of $D$ in general.
Hence, exceptional domains should be understood as pointed
domains $p\in D$.

%
%
The simplest noncircular example in the above family of domains
is obtained by taking $a=0$, $\lambda=\frac12$, and $c=1$.
This gives
\[
	F'(z)=\frac{1}{1+z^2/2},	\qquad
	F(z)=\sqrt2\,\arctan\!\left(\frac{z}{\sqrt2}\right);
\]
see Corollary \ref{cor:counterexample}.
The image \(F(\mathbb D)\) is not a disc, which
gives a negative answer to \cite[Problem~4.3]{ForstnericKalaj2024}.

%
%
Our next result is a converse to Theorem
\ref{th:quadratic-exceptional-family}
in the class of bounded convex domains. The two results together
classify exceptional pointed domains among bounded convex domains.

%
%
\begin{theorem}\label{th:general-convex}
Let \(\Phi:\mathbb D\to D\) be a biholomorphic map onto a
bounded convex domain with $\Phi(0)=0$
and $|\Phi'(0)|=M_D(0)$. Then, up to a rotation, $\Phi$ equals
a map $F$ from Theorem \ref{th:quadratic-exceptional-family}.
\end{theorem}

Theorem \ref{th:general-convex} is proved in Sect.\ \ref{sec:equality}.
The proof is based on the boundary value formulation of the
harmonic extremal problem.

It would be interesting to consider the analogous problem
for harmonic maps $\mathbb D\to D$ to domains
$D\subset \mathbb R^n$ for $n\ge 3$.
Let $z=x+iy$ be the coordinate on $\mathbb C$.
Given a point $p\in D$ and a tangent vector
$v\in \mathbb R^n\setminus \{0\}$, one defines
$M_D(p,v)$ as the supremum of the numbers $\|df_0\|$
over all harmonic maps $f:\mathbb D\to D$ with
$f(0)=p$, $f_x(0)=\lambda v$ for some $\lambda> 0$,
and $df_0:\mathbb R^2\to\mathbb R^n$ a conformal linear map.

\begin{problem}
For which bounded convex pointed domains
$p\in D\subset \mathbb R^n$, $n\ge 3$,
and vectors $v\in \mathbb R^n$ is the supremum
$M_D(p,v)$ reached by a conformal harmonic map
$f:\mathbb D\to D$?
\end{problem}

By \cite[Theorem 2.1]{ForstnericKalaj2024}, the answer is affirmative when
$p\in D=\mathbb B^n$ for any $v\in \mathbb R^n\setminus\{0\}$.

%
%
%
%
\section{Existence of extremal maps}\label{sec:existence}

Recall that for a domain $D\subset \mathbb C$
and a point $p\in D$, $M_D(p)$ is defined by \eqref{eq:MD}.
By $\operatorname{Per}(D)$ we denote the perimeter of a
bounded domain $D$. In this section we prove the following result.

\begin{theorem}\label{th:existence}
Let \(D\subset\mathbb C\) be a bounded convex domain and
\(p\in D\). There exists a harmonic map \(f:\mathbb D\to D\)
such that $f(0)=p$ and $\|df_0\|=M_D(p)$. Furthermore,
\[
	M_D(p)\le \frac{\operatorname{Per}(D)}{2\pi}.
\]
\end{theorem}

\begin{proof}
Let $z=x+iy$ be the complex coordinate on $\mathbb C$.
Given a differentiable function $f$ on a domain in $\mathbb C$,
we denote
\[
	f_z=\frac12(f_x-if_y), \qquad f_{\bar z}=\frac12(f_x+if_y).
\]
Then, $df = f_z dz + f_{\bar z}d\bar z$.
The condition $f_{\bar z}(z)=0$ means that $f$ is conformal and orientation
preserving at $z$, and in this case, $f_z(z)=f'(z)$ is the complex
derivative of $f$ at $z$.

By a translation of $D$ we may assume that $p=0$.
Let \(\mathcal A_D\) denote the set of all functions
\(\psi\in L^\infty([0,2\pi])\) such that
\[
	\psi(t)\in \bar{D} \quad \text{for a.e. }t
\]
and
\[
	\int_0^{2\pi}\psi(t)\,dt=0,
	\qquad
	\int_0^{2\pi} e^{it}\psi(t)\,dt=0.
\]
We may think of \(\psi\in\mathcal A_D\) as a function
on the circle $\partial \mathbb D=\{e^{it}:t\in [0,2\pi]\}$.
Let \(f=P[\psi]:\mathbb D\to \mathbb C\)
denote its Poisson extension. Then \(f\) is harmonic, and the above
conditions on $\psi$ imply
\[
	f(0)=0,
	\qquad
	f_{\bar z}(0)=0.
\]
Hence \(df_0\) is conformal and orientation-preserving.
Conversely, after replacing \(f(z)\) by \(f(\bar z)\), 
if necessary, every admissible harmonic map $\mathbb D\to D$
arises in this way.

After precomposing with a rotation of \(\mathbb D\),
we may assume that \(f_z(0)>0\), so
\[
	\|df_0\|=f_z(0).
\]
Thus
\[
	M_D(0)=\sup_{\psi\in\mathcal A_D}\Re L(\psi),
	\qquad
	L(\psi):=f_z(0)=\frac1{2\pi}\int_0^{2\pi} e^{-it}\psi(t)\,dt.
\]
Let \(h_D\) be the support function of \(D\),
\[
	h_D(t):=\sup_{w\in D}\Re(e^{-it}w),\quad t\in [0,2\pi].
\]
For every \(\psi\in\mathcal A_D\) we have \(\psi(t)\in \bar{D}\) for
almost every \(t\), and hence
\[
	\Re(e^{-it}\psi(t))\le h_D(t)
	\qquad\text{for a.e. }t.
\]
Therefore,
\[
	\Re L(\psi)
	=
	\frac1{2\pi}\int_0^{2\pi}\Re(e^{-it}\psi(t))\,dt
	\le
	\frac1{2\pi}\int_0^{2\pi} h_D(t)\,dt.
\]
Taking the supremum over \(\psi\in\mathcal A_D\) gives
\[
	M_D(0)\le \frac1{2\pi}\int_0^{2\pi} h_D(t)\,dt.
\]
By Cauchy's perimeter formula for bounded convex planar domains
(see \cite{Santalo1976}),
\[
	\int_0^{2\pi} h_D(t)\,dt=\operatorname{Per}(D),
\]
and hence
\[
	M_D(0)\le \frac{\operatorname{Per}(D)}{2\pi}.
\]

Since \(D\) is bounded, \(\mathcal A_D\) is contained in a weak-*
compact ball of \(L^\infty\). Because \(\bar D\) is convex and closed,
it is an intersection of closed halfplanes,
so \(\mathcal A_D\) is weak-* closed and hence weak-* compact.
Since \(\Re L\) is weak-* continuous,
there exists \(\psi_0\in\mathcal A_D\) such that
\[
	M_D(0)=\Re L(\psi_0).
\]
(The existence of such $\psi_0$ can also be seen by a
normal families argument.)
Let \(f_0=P[\psi_0]\) denote its Poisson integral.
We claim that \(f_0(\mathbb D)\subset D\). Let
\[
	H_{\beta,c}:=\{w\in\mathbb C:\ \Re(e^{-i\beta}w)<c\}
\]
be any open supporting halfplane of \(D\).
Since \(0\in D\), we have \(c>0\). Since the a.e.\ boundary values of
\(f_0\) lie in \(\overline D\subset \overline{H_{\beta,c}}\), the harmonic
function
\[
	u(z):=\Re(e^{-i\beta}f_0(z))-c
\]
satisfies \(u\le 0\) on \(\mathbb D\), while
\[
	u(0)=\Re(e^{-i\beta}\cdot 0)-c=-c<0.
\]
By the maximum principle, \(u<0\) on \(\mathbb D\).
Since this holds for every supporting halfplane
of \(D\), we conclude that
\[
	f_0(\mathbb D)\subset D.
\]
Therefore \(f_0\) is admissible and \(\|df_0\|=M_D(0)\).
\end{proof}

%
%
%
%
\section{Exceptional domains and a
noncircular example}\label{sec:exceptional}

In this section we prove Theorem \ref{th:quadratic-exceptional-family}.
We begin with preliminaries. We shall use the following form of the
classical Schur--Cohn criterion \cite{Schur1918,CohnA1922}.

\begin{lemma}\label{lem:schur-cohn-quadratic}
The quadratic holomorphic map
\begin{equation}\label{eq:q}
	q(z)=1+az+\lambda z^2,
	\qquad a,\lambda\in\mathbb C
\end{equation}
has no zeros on \(\overline{\mathbb D}\) if and only if
\begin{equation}\label{eq:SC1}
	|\lambda|<1 \quad \text{and} \quad	
	|a-\bar a\lambda|<1-|\lambda|^2 .
\end{equation}
\end{lemma}

\begin{proof}
Put
\[
	P(w)=w^2+aw+\lambda.
\]
If \(z\ne0\) and \(w=1/z\), then
\[
	q(z)=0
	\quad\Longleftrightarrow\quad
	P(w)=0.
\]
Thus \(q\) has no zeros on \(\overline{\mathbb D}\) if and only if both zeros
of \(P\) lie in \(\mathbb D\). For the monic quadratic
polynomial $P(w)=w^2+aw+\lambda$,
the Schur--Cohn criterion is precisely \eqref{eq:SC1}.
\end{proof}

\begin{lemma}\label{lem:quadratic-convexity}
Assume that the numbers \(a,\lambda\in\mathbb C\) satisfy
\[
	|\lambda|<1
	\quad \text{and}\quad
	|a-\bar a\lambda| \le 1-|\lambda|^2,
\]
and let $F:\mathbb D\to\mathbb C$ be such that
\[
 	F'(z)=\frac{c}{1+az+\lambda z^2},\qquad c>0.
\]
Then \(F\) maps \(\mathbb D\) biholomorphically onto the convex
domain $D=F(\mathbb D)$. If in addition
\[
	|a-\bar a\lambda| < 1-|\lambda|^2,
\]
then $D=F(\mathbb D)$ is a bounded strongly convex domain.
\end{lemma}

\begin{proof}
By the non-strict form of the Schur--Cohn criterion, the assumptions
on $a$ and $\lambda$ imply that the polynomial
$1+az+\lambda z^2$ has no zeros in $\mathbb D$. 
Hence, $F'$ is holomorphic
and nonzero in $\mathbb D$. By the standard convexity criterion for
holomorphic functions, it suffices to prove that
\begin{equation}\label{eq:F''}
	\Re\!\left(1+z\frac{F''(z)}{F'(z)}\right)\ge 0,
	\qquad z\in\mathbb D.
\end{equation}
A computation gives
\begin{equation}\label{eq:pz}
		p(z):= 1+z\frac{F''(z)}{F'(z)}
	= \frac{1-\lambda z^2}{1+az+\lambda z^2}.
\end{equation}
Set
\[
	\omega(z):=\frac{p(z)-1}{p(z)+1}=-\,\frac{z(a+2\lambda z)}{2+az}.
\]
The assumption
$
	|a-\bar a\lambda|\le 1-|\lambda|^2
$
implies
\[
	|a|(1-|\lambda|)
	\le |a-\bar a\lambda|
	\le 1-|\lambda|^2.
\]
It follows that $|a|\le 1+|\lambda|<2$ and hence
$2+az\neq0$ for $|z|\le1$, thereby showing that \(\omega\) is holomorphic
in a neighbourhood of the closed disc \(\overline{\mathbb D}\).
We claim that
\begin{equation}\label{eq:omega}
	 |\omega(z)| < 1,
	 \quad z\in \mathbb D.
\end{equation}	
For \(|z|=1\) we have
\[
	|\omega(z)|\le1
	\iff
	|a+2\lambda z|\le |2+a z|.
\]
After squaring and expanding, we have
\[
	|2+\bar a z|^2-|a+2\lambda z|^2
	=
	4(1-|\lambda|^2)
	+
	4\Re \bigl((a-\bar a\lambda)z\bigr).
\]
By the assumption
$
	|a-\bar a\lambda|\le 1-|\lambda|^2,
$
the right-hand side is nonnegative, and hence $|\omega(z)|\le1$ for $|z|=1$.
By the maximum principle, $|\omega(z)|\le 1$ for all $z\in \bar{\mathbb D}$.
Since \(\omega(0)=0\), the function \(\omega\) is not identically equal to a
unimodular constant, which implies \eqref{eq:omega}. Therefore, the function
\[
	p(z)=\frac{1+\omega(z)}{1-\omega(z)},\qquad z\in \mathbb D,
\]
has positive real part in \(\mathbb D\). In view of \eqref{eq:pz}, this
establishes \eqref{eq:F''}, so \(F\) is convex in \(\mathbb D\). If
\[
	|a-\bar a\lambda|<1-|\lambda|^2,
\]
then the boundary inequality above is strict:
\[
	|\omega(z)|<1,
	\qquad |z|=1.
\]
By the standard strict convexity criterion for analytic functions,
\(F\) maps \(\mathbb D\) conformally onto a bounded
strongly convex domain.
\end{proof}

%
%
\begin{proof}[Proof of Theorem~\ref{th:quadratic-exceptional-family}]
By Lemma~\ref{lem:schur-cohn-quadratic}, the assumption that
$
	q(z)=1+az+\lambda z^2
$
has no zeros on \(\overline{\mathbb D}\) is equivalent to
\[
|\lambda|<1,
\qquad
|a-\bar a\lambda|<1-|\lambda|^2 .
\]
In particular, \(F'\), and hence \(F\), extends holomorphically to a
neighbourhood of \(\overline{\mathbb D}\).
By Lemma~\ref{lem:quadratic-convexity}, \(F\) maps \(\mathbb D\) conformally onto a bounded strongly convex domain
$D=F(\mathbb D)$.

We now prove the extremal property $|F'(0)|=M_D(0)$.
To this end, we must show that if $f:\mathbb D\to D$ is 
a harmonic map with $f(0)=0$ such that \(df_0\) is conformal,
then $\|df_0\|\le c=|F'(0)|$. If \(df_0=0\), there is nothing to prove. 
If \(df_0\) is orientation-reversing, we replace \(f(z)\) by \(f(\bar z)\); 
this does not change \(\|df_0\|\). Thus, we may assume that
$
	f_{\bar z}(0)=0.
$
Precomposing with a rotation of \(\mathbb D\), we may also assume that
$
	f_z(0)>0.
$
Then
\[
	\|df_0\|=f_z(0).
\]
%
Since \(D\) is bounded, \(f\) has radial boundary values
$
	\psi:\partial\mathbb D\to \overline D
$
almost everywhere, and \(f=P[\psi]\) is the Poisson integral of $\psi$.
By the Poisson formula,
\[
	f(0)=\frac1{2\pi}\int_0^{2\pi}\psi(e^{it})\,dt=0
\]
and
\[
	f_z(0)
	=
	\frac1{2\pi}\int_0^{2\pi}e^{-it}\psi(e^{it})\,dt,
	\qquad
	f_{\bar z}(0)
	=
	\frac1{2\pi}\int_0^{2\pi}e^{it}\psi(e^{it})\,dt.
\]
Therefore
\[
	\|df_0\|
	=
	f_z(0)
	=
	\frac1{2\pi}
	\int_0^{2\pi}e^{-it}\psi(e^{it})\,dt.
\]
Using the constraints
\[
	\int_0^{2\pi}\psi(e^{it})\,dt=0,
	\qquad
	\int_0^{2\pi}e^{it}\psi(e^{it})\,dt=0,
\]
we may add the corresponding zero terms and obtain
\begin{equation}\label{eq:estdf0}
	\|df_0\|
	= f_{z}(0)
	=
	\frac1{2\pi} \int_0^{2\pi}
	\bigl(e^{-it}+a+\lambda e^{it}\bigr)\psi(e^{it})\,dt
	=
	\frac1{2\pi} \int_0^{2\pi}
	\bigl( V(t)\psi(e^{it})\,dt
\end{equation}
where
\begin{equation}\label{eq:V}
	V(t):=e^{-it}+a+\lambda e^{it}
	= e^{-it} q(e^{it})\ne 0,
	\qquad t\in\mathbb R.
\end{equation}
Here, $q(z)=1+az+\lambda z^2$ is as in \eqref{eq:q}.
Recall that $q$ has no zeros on $\bar{\mathbb D}$.
We claim that, for every \(t\in\mathbb R\), the real-linear functional
\[
	w\longmapsto \Re\bigl(V(t)w\bigr)
\]
is uniquely maximized over \(\overline D\) at \(w=F(e^{it})\). Indeed,
\eqref{eq:F} gives
$
	q(e^{it})F'(e^{it})=c,
$
and hence
\[
	e^{it}F'(e^{it})
	=
	\frac{c\,e^{it}}{q(e^{it})}.
\]
From this and \eqref{eq:V} we obtain
\[
	\frac{\overline{V(t)}}{e^{it}F'(e^{it})}
	= e^{it} \, \overline{q(e^{it})}\,
	    \frac{q(e^{it})}{c\, e^{it}}
	= \frac{1}{c} |q(e^{it})|^2 >0.
\]
Thus for every $t$, \(\overline{V(t)}\) is a positive real multiple of
$e^{it}F'(e^{it})$, which is the outward normal direction to the
positively oriented boundary curve
$
	t\mapsto F(e^{it}) \in bD.
$
Since \(D\) is strongly convex, the supporting functional
$
	w\mapsto \Re\bigl(V(t)w\bigr)
$
has the unique maximizer
$
	w=F(e^{it})
$
on \(\overline D\) as claimed. Therefore, for almost every \(t\),
\[
	\Re\bigl(V(t)\psi(e^{it})\bigr)
	\le
	\Re\bigl(V(t)F(e^{it})\bigr).
\]
In view of \eqref{eq:estdf0}, it follows that
\begin{equation}\label{eq:estimate}
	\|df_0\|
	\le
	\frac1{2\pi}
	\Re\int_0^{2\pi}V(t)F(e^{it})\,dt.
\end{equation}
Note that
\[
	\int_0^{2\pi}F(e^{it})\,dt= 2\pi F(0)=0.
\]
Also, since \(F\) is holomorphic in a neighbourhood
of \(\overline{\mathbb D}\),
\[
	\int_0^{2\pi}e^{it}F(e^{it})\,dt=0.
\]
Hence
\[
	\frac1{2\pi}
	\Re\int_0^{2\pi}V(t)F(e^{it})\,dt
	=
	\frac1{2\pi}
	\Re\int_0^{2\pi}e^{-it}F(e^{it})\,dt
	=
	F'(0)
	=
	c.
\]
From this and \eqref{eq:estimate} we conclude that
\[
	\|df_0\|\le c.
\]
Since \(F\) itself is admissible and
$
	\|dF_0\|=F'(0)=c,
$
it follows that
\[
	M_D(0)=c.
\]
Equality $\|df_0\|=c$ can hold only if the boundary value function
$\psi$ of $f$ satisfies
\[
	\psi(e^{it})=F(e^{it})
	\qquad\text{for a.e. }t,
\]
because, for every \(t\), the functional
$
	w\mapsto \Re\bigl(V(t)w\bigr)
$
has a unique maximizer $w=F(e^{it})$ over \(\overline D\).
Hence, by uniqueness of the Poisson extension, the only extremal
satisfying
\[
	f_{\bar z}(0)=0,
	\qquad
	f_z(0)>0
\]
is
$
	f=F.
$
Undoing rotations of the source, the orientation-preserving extremals are
precisely
\[
	F\circ R_\theta,
	\qquad
	R_\theta(z)=e^{i\theta}z,
	\quad \theta\in\mathbb R.
\]
If orientation-reversing conformal differentials are also allowed, then applying
the preceding argument to \(f(\bar z)\) gives the additional extremals
$
	z\mapsto F(e^{i\theta}\bar z).
$
\end{proof}

As an immediate consequence, we obtain a negative answer to
\cite[Problem~4.3]{ForstnericKalaj2024}.

%
%
\begin{corollary}\label{cor:counterexample}
There exists a bounded strongly convex domain
\(0\in D\subset\mathbb C\), which is not a disc, such that
the supremum of \(\|df_0\|\) over all harmonic maps
\[
	f:\mathbb D\to D,
	\qquad f(0)=0,
\]
whose differential is conformal at \(0\), is attained by a biholomorphic
map \(F:\mathbb D\to D\) with $F(0)=0$.
Moreover, among orientation-preserving extremals,
\(F\) is unique up to precomposition by rotations of \(\mathbb D\);
after the normalization \(F_z(0)>0\), it is unique.
\end{corollary}

\begin{proof}
Choose
\[
	\lambda=\frac12,
	\qquad a=0,
	\qquad c=1
\]
in Theorem~\ref{th:quadratic-exceptional-family}. Then
\[
	F(0)=0,
	\qquad
	F'(z)=\frac{1}{1+z^2/2},
\]
and hence
\begin{equation}\label{eq:actan}
	F(z)=\sqrt2\,\arctan\!\left(\frac{z}{\sqrt2}\right).
\end{equation}
By Theorem~\ref{th:quadratic-exceptional-family}, the image
$D=F(\mathbb D)$ is a bounded strongly convex domain, and
$
	M_D(0)=|F'(0)|=1.
$
Moreover, the extremals are precisely
those described in Theorem~\ref{th:quadratic-exceptional-family}.
If \(D\) were a disc, then the map \(F:\mathbb D\to D\)
would be a M\"obius transformation. Since the map \eqref{eq:actan}
is not a M\"obius transformation, \(D\) is not a disc.
\end{proof}

%
%
%
%
\section{Bounded convex domains: classification of equality}\label{sec:equality}

In this section we prove Theorem \ref{th:general-convex}.

\begin{proof}
Let \(\Phi:\mathbb D\to D\) be a biholomorphic map onto a
bounded convex domain $D\subset \mathbb C$ with $\Phi(0)=0$.
Rotating the target $D=\Phi(\mathbb D)$,
we may assume that $\Phi'(0)>0$. Our goal is to prove
that $\Phi$ equals a map $F$ in \eqref{eq:F} from
Theorem \ref{th:quadratic-exceptional-family}.

Since \(D\) is bounded and convex, \(\partial D\) is a rectifiable Jordan curve.
Hence, \(\Phi\) extends homeomorphically to \(\overline{\mathbb D}\),
\(\Phi'\in H^1\), and its boundary value function
\begin{equation}\label{eq:varphi}
	\varphi(t)=\psi_0(e^{it}) := \Phi(e^{it}),
	\quad t\in\mathbb R,
\end{equation}
is absolutely continuous. Moreover,
\begin{equation}\label{eq:phidot}
	\dot\varphi(t)= i e^{it}\Phi'(e^{it})
	\qquad\text{for a.e. }t.
\end{equation}
We formulate the extremal problem in terms of the boundary function. Let
\begin{equation}\label{eq:K}
	K:=
	\left\{
	\psi\in L^\infty(\partial\mathbb D, \mathbb C):
	\psi(e^{it})\in\overline D \text{ for a.e. }t
	\right\}.
\end{equation}
For \(\psi\in K\), define
\[
	A_0(\psi):=\frac1{2\pi}\int_0^{2\pi}\psi(e^{it})\,dt,
	\qquad
	A_1(\psi):=\frac1{2\pi}\int_0^{2\pi}e^{it}\psi(e^{it})\,dt,
\]
and
\[
	J(\psi):=
	\Re\frac1{2\pi}\int_0^{2\pi}e^{-it}\psi(e^{it})\,dt.
\]
Let
\[
	K_0:=\{\psi\in K:\ A_0(\psi)=0,\ A_1(\psi)=0\}.
\]
Note that $K_0$ equals the set $\mathcal A_D$ introduced in
Sect.\ \ref{sec:existence}. If \(\psi\in K_0\), its Poisson extension
\(f=P[\psi]:\mathbb D\to\mathbb C\) satisfies
$f(0)=0$, $f_{\bar z}(0)=0$, and hence \(df_0\) is an orientation
preserving conformal linear map. Moreover,
\[
	f_z(0)=\frac1{2\pi}\int_0^{2\pi}e^{-it}\psi(e^{it})\,dt.
\]
By the argument in the proof of Theorem \ref{th:existence},
we have \(f(\mathbb D)\subset D\) since $D$ is convex.

Let $\Phi$ be as above. Its boundary map
$
	\psi_0:=\Phi|_{\partial\mathbb D}
$
belongs to \(K_0\), and
\[
	J(\psi_0)=\Phi'(0).
\]
Since \(M_D(0)=\Phi'(0)\), the function \(\psi_0\) maximizes \(J\) over \(K_0\).

We now apply a finite-dimensional separation argument.
Note that \(K\) is convex because \(\overline D\) is convex:
if \(\psi_1,\psi_2\in K\) and \(0\le \theta\le1\), then
\[
	\theta\psi_1(e^{it})+(1-\theta)\psi_2(e^{it})\in \overline D
	\qquad\text{for a.e. }t.
\]
Moreover, \(A_0\), \(A_1\), and \(J\) are real-linear functionals on $K$.
Hence the image
\[
	\mathcal S=\{(A_0(\psi),A_1(\psi),J(\psi)):\psi\in K\}
	\subset \mathbb C^2\times\mathbb R
\]
is convex in \(\mathbb C^2\times\mathbb R\). Let
\[
	s_0:=(0,0,J(\psi_0)).
\]
Since \(\psi_0\) maximizes \(J\) on \(K_0\), the open vertical ray
\begin{equation}\label{eq:R}
	\mathcal R:=\{(0,0,s):s>J(\psi_0)\}
\end{equation}
is disjoint from \(\mathcal S\).
Moreover, \((0,0)\in\mathbb C^2\) is an interior point of the projection
of \(\mathcal S\) to \(\mathbb C^2\). Indeed, since \(0\in D\),
there is \(\delta>0\) such that
$\{w\in\mathbb C:|w|<\delta\}\subset D$.
Hence, whenever \(|u|+|v|<\delta\), the function
$
	\psi_{u,v}(e^{it})=u+v e^{-it}
$
belongs to \(K\), and $A_0(\psi_{u,v})=u$, $ A_1(\psi_{u,v})=v$.

Since \(\mathcal S\) is convex and the convex set
$\mathcal R$ in \eqref{eq:R} is disjoint from \(\mathcal S\),
the finite-dimensional separation theorem
(see e.g.\ \cite[Sect.\ 2.5]{BoydVandenberghe2004})
gives a nonzero real-linear functional
\begin{equation}\label{eq:ell}
	\ell(z_0,z_1,s)=\Re(\alpha z_0+\beta z_1)+\mu s
\end{equation}
on $\mathbb C^2\times\mathbb R$ and a real number \(\gamma\) such that
\[
	\ell(x)\le \gamma\le \ell(y)
	\qquad
	(x\in\mathcal S,\ y\in\mathcal R).
\]
Since \(s_0=(0,0,J(\psi_0))\in\mathcal S\), we have
\[
	\ell(s_0)\le \gamma.
\]
On the other hand, letting \(y=(0,0,s)\in\mathcal R\) and then
\(s\downarrow J(\psi_0)\), we obtain
\[
	\gamma\le \ell(s_0).
\]
Therefore
$
	\ell(s_0)=\gamma,
$
so the separating hyperplane supports \(\mathcal S\) at \(s_0\).

We claim that \(\mu>0\) in \eqref{eq:ell}. First, since
\[
	\ell(0,0,s)=\mu s\ge \gamma
	\qquad (s>J(\psi_0)),
\]
we must have \(\mu\ge0\). If \(\mu=0\), then \(\ell\) depends only on
\((z_0,z_1)\). But \(0\) is an interior point of the projection of
\(\mathcal S\) onto \(\mathbb C^2\), and a nonzero real-linear functional
cannot support a set at an interior point of its projection. Hence
\(\alpha=\beta=0\), contradicting the fact that \(\ell\) is nonzero. Thus
$\mu>0$ as claimed. Dividing by \(\mu\), we obtain constants
\(a,\lambda\in\mathbb C\) such that
\(
	\psi_0
\)
maximizes the functional
\[
	K\ni \psi\longmapsto
	\Re\frac1{2\pi}\int_0^{2\pi}
	\bigl(e^{-it}+a+\lambda e^{it}\bigr)\psi(e^{it})\,dt
	= \Re\frac1{2\pi}\int_0^{2\pi} V(t)\psi(e^{it})\,dt
\]
where 
\[
	V(t):=e^{-it}+a+\lambda e^{it} 
\]
is as in \eqref{eq:V}. Note that
\begin{equation}\label{eq:qeit}
	q(e^{it}) := e^{it}V(t)=1+a e^{it}+\lambda e^{2it}
\end{equation}
is a nonzero polynomial of degree at most two in \(e^{it}\), 
and hence \(V(t)\) has at most two zeros on \(t\in [0,2\pi)\).
Since the set \(K\) in \eqref{eq:K} is defined by pointwise
conditions $\psi(e^{it})\in\overline D$ for a.e.\ $t$, 
the fact that $\psi_0$ maximizes this integral implies, 
for almost every $t$, the pointwise support condition
\[
	\Re\bigl(V(t)w\bigr)
	\le
	\Re\bigl(V(t)\psi_0(e^{it})\bigr)
	\qquad (w\in\overline D).
\]
Since $V$, $\varphi$, and the support function of $\overline D$ are
continuous, this inequality holds for all $t$. Writing
$\varphi(t)=\psi_0(e^{it})$ (see \eqref{eq:varphi}), the support condition is
equivalent to
\[
	\Re\bigl(V(t)(w-\varphi(t))\bigr)\le0
	\qquad (w\in\overline D).
\]
Taking \(w=\varphi(t+h)\) and using both positive and negative
values of \(h\), at every
point $t$ where \(\varphi\) is differentiable (which holds for
almost every $t$, see \eqref{eq:phidot}) we get
\[
	\Re\bigl(V(t)\dot\varphi(t)\bigr)=0.
\]
In view of \eqref{eq:phidot}, this gives
\[
	\Re\bigl(V(t)\, i e^{it}\Phi'(e^{it})\bigr)=0
	\qquad\text{for a.e. }t.
\]
By \eqref{eq:qeit}, this is equivalent to
\begin{equation}\label{eq:Im}
	\Im\Bigl(\bigl(1+a e^{it}+\lambda e^{2it}\bigr)\Phi'(e^{it})\Bigr)=0
	\qquad\text{for a.e. }t.
\end{equation}
Set
\[
	G(z):=(1+a z+\lambda z^2)\Phi'(z) = q(z) \Phi'(z).
\]
Since \(\partial D\) is rectifiable, we have \(\Phi'\in H^1\) and hence
$G\in H^1$. By \eqref{eq:Im}, the harmonic function
\(\Im G\) has nontangential boundary values equal to \(0\)
almost everywhere. By uniqueness for harmonic functions in \(h^1\),
$\Im G\equiv0$. Hence \(G\) is constant,
$G(z)\equiv c=G(0)=\Phi'(0)>0$, and
\[
	\Phi'(z)=\frac{c}{q(z)} = \frac{c}{1+a z+\lambda z^2}.
\]
Since \(q\Phi'=c>0\) in \(\mathbb D\), the polynomial \(q\)
has no zeros in \(\mathbb D\). We claim that \(q\) has no zeros on
\(\partial\mathbb D\) either. Suppose to the contrary that
$q(\zeta)=0$ for some $\zeta\in \partial\mathbb D$.
Since \(q\) has degree at most two, \(\zeta\) is either a simple or a double
zero. If \(\zeta\) is a simple zero, then \(c/q(z)\) has a simple pole at \(\zeta\).
Integrating along the radius \(z=r\zeta\), \(r\uparrow1\), gives logarithmic
growth of \(\Phi(r\zeta)\), contradicting the boundedness of
the domain $\Phi(\mathbb D)=D$.
If \(\zeta\) is a double zero, then \(c/q(z)\) has a double pole at \(\zeta\).
Integrating along the same radius 
shows that \(\Phi(r\zeta)\) has pole-type growth,
again contradicting the boundedness of \(D\).
Thus, $q$ has no zeros on \(\overline{\mathbb D}\),
and $\Phi$ is of the form \eqref{eq:F} in
Theorem \ref{th:quadratic-exceptional-family}.
\end{proof}

\begin{remark}
No smoothness assumption is made in
Theorem~\ref{th:general-convex}; smoothness
and strict convexity are deduced from the equality assumption.
\end{remark}

%
%
\subsection*{Acknowledgements}
F.\ Forstneri{\v c} was supported by the European Union
(ERC Advanced grant HPDR, 101053085) and by grant P1-0291
from ARIS, Republic of Slovenia.
D.~Kalaj gratefully acknowledges financial support from the Ministry of
Education, Science and Innovation of Montenegro through the grants
\emph{``Mathematical Analysis, Optimisation and Machine Learning''}
and \emph{``Complex-analytic and geometric techniques for non-Euclidean
machine learning: theory and applications.''}

%
%
%
%
\subsection*{Declaration on the use of AI} 
The mathematical ideas, statements, and proof strategies in this article 
are due to the authors. Generative AI tools (ChatGPT, OpenAI)
were used to assist with drafting the technical arguments and
language editing.  
The authors subsequently checked and revised all such material and 
take full responsibility for the correctness and presentation of the article.

\end{document}